\documentclass[10pt]{amsart}
\usepackage{amssymb}
\usepackage{epsfig}
\usepackage{url}
\usepackage{setspace}
\theoremstyle{plain}

\newtheorem{thm}{Theorem}[section]
\newtheorem{cor}[thm]{Corollary}
\newtheorem{lem}[thm]{Lemma}

\newtheorem{conj}[thm]{Conjecture}

\newtheorem{prob}[thm]{Problem}
\def\cal{\mathcal}
\def\bbb{\mathbb}
\def\op{\operatorname}
\renewcommand{\phi}{\varphi}
\newcommand{\R}{\bbb{R}}
\newcommand{\N}{\bbb{N}}
\newcommand{\Z}{\bbb{Z}}
\newcommand{\Q}{\bbb{Q}}

\begin{document}

\title[Stern polynomials]{On certain arithmetic properties of Stern polynomials}
\author{Maciej Ulas and Oliwia Ulas}
\thanks{The first named author is holder of START scholarship funded by the Foundation
for Polish Science (FNP)} 

\begin{abstract}
We prove several theorems concerning arithmetic properties of Stern
polynomials defined in the following way: $B_{0}(t)=0, B_{1}(t)=1,
B_{2n}(t)=tB_{n}(t)$, and $B_{2n+1}(t)=B_{n}(t)+B_{n+1}(t)$. We
study also the sequence $e(n)=\op{deg}_{t}B_{n}(t)$ and give various
properties of it.

\end{abstract}

\maketitle

\section{Introduction}\label{sec1}

The {\it Stern sequence} (or {\it Stern's diatomic sequence}) $s(n)$
was introduced in \cite{Ste} and is defined recursively in the
following way
\begin{equation*}
s(0)=0,\quad s(1)=1,\quad s(n)=
\begin{cases}
\begin{array}{lll}
  s(\frac{n}{2})                     & & \mbox{for}\;n\;\mbox{even},  \\
  s(\frac{n-1}{2})+s(\frac{n+1}{2}) & & \mbox{for}\;n\;\mbox{odd}.
\end{array}
\end{cases}
\end{equation*}
This sequence appears in different mathematical contexts. For
example in \cite{Hin} a pure graph theoretical problem is considered
related to the metric properties of the so-called Tower of Hanoi
graph. In the cited paper it is shown that the Stern sequence
appears in the counting function of certain paths in this graph.

In the paper \cite{Rez} $s(n)$ appears as the number of partitions
of a natural number $n-1$ in the form
$n-1=\sum_{i=0}^{\infty}\epsilon_{i}2^{i}$, where
$\epsilon_{i}\in\{0,1,2\}$. These are called hiperbinary
representations. The connections of the Stern sequence with
continued fractions and the Euclidean algorithm are considered in
\cite{Leh} and \cite{Lin}. An interesting application of the Stern
sequence to the problem of construction a bijection between $\N_{+}$
and $\Q_{+}$ is given in \cite{Cal:Wil}. In this paper it is shown
that the sequence $s(n)/s(n+1)$, for $n\geq 1$, encounters every
positive rational number exactly once.

A comprehensive survey of properties of the Stern sequence can be
found in \cite{Urb}. An interesting survey of known results and
applications of the Stern sequence can also be found in \cite{Nor}.

Recently two distinct polynomial analogues of the Stern sequence
appeared. The sequence of polynomials $a(n;x)$ for $n\geq 0$,
defined by $a(0;x)=0$, $a(1;x)=1$, and for $n\geq 2$:
\begin{equation*}
a(2n;x)=a(n;x^2),\quad a(2n+1;x)=xa(n;x^2)+a(n+1;x^2),
\end{equation*}
was considered in \cite{DilSto}. It is easy to see that
$a(n;1)=s(n)$. Remarkably, as was proved in the cited paper,
$xa(2n-1;x)\equiv A_{n+1}(x)\pmod{2}$, where
$A_{n}(x)=\sum_{j=0}^{n}S(n,j)x^{j}$ and $S(n,j)$ are the Stirling
numbers of the second kind. Further properties of this sequence and
its connection with continued fractions can be found in
\cite{DilSto2}.

Let us consider the sequence of Stern polynomials $B_{n}(t)$, $n\geq
0$, defined recursively in the following way:
\begin{equation*}
B_{0}(t)=0,\quad B_{1}(t)=1,\quad B_{n}(t)=
\begin{cases}
\begin{array}{lll}
  tB_{\frac{n}{2}}(t)                     & & \mbox{for}\;n\;\mbox{even},  \\
  B_{\frac{n-1}{2}}(t)+B_{\frac{n+1}{2}}(t) & & \mbox{for}\;n\;\mbox{odd}.
\end{array}
\end{cases}
\end{equation*}
This sequence of polynomials was introduced in \cite{Kla} and is the
sequence which we investigate in this paper. In \cite{Kla} it is
shown that the sequence of Stern polynomials has an interesting
connections with some combinatorial objects.

The aim of this paper is to give some arithmetic properties which
can be deduced from the definition of the sequence of Stern
polynomials.

In Section \ref{sec2} we gather basic properties of the sequence of
Stern polynomials. In particular in the Theorem \ref{symprop} we
prove a symmetric property of $B_{n}(t)$ (i. e., generalization of
the property $s(i)=s(2^{n}-i)$ for $1\leq i\leq 2^{n}-1$). Among
other things we also prove that for each $n\in\N_{+}$ the
polynomials $B_{n}(t), B_{n+1}(t)$ are coprime (Corollary
\ref{gcd}).

In Section \ref{sec3} we consider the generating function of the
sequence of Stern polynomials (Theorem \ref{genfunt}). With its use
we give various identities between Stern polynomials and show that
the sequence of the degrees of Stern polynomials are connected with
the sequence $\nu(n)$ which counts the occurrence of 1's in the
binary representation of the number $n$ (Corollary \ref{corl}).

In Section \ref{sec4} we investigate the properties of the sequence
$\{e(n)\}_{n=1}^{\infty}$, where $e(n)=\op{deg}_{t}B_{n}(t)$, which
is interesting in its own. In particular we compute the exact number
of Stern polynomials with degree equal to $n$. In order to prove
desired result we use the generating function of the sequence
$\{e(n)\}_{n=1}^{\infty}$ (Theorem \ref{genfundeg}). We also
investigate the extremal properties of the sequence
$\{e(n)\}_{n=1}^{\infty}$ (Theorem \ref{extremalofe}).

In Section \ref{sec5} we investigate special values of the
polynomial $B_{n}(t)$ and give some applications to the diophantine
equations of the form $B_{n+a}(t)-B_{n}(t)=c$, where $a\in\N$ is
fixed. Section \ref{sec6} is devoted to open problems and
conjectures which appear during our investigations and which we were
unable to prove.

\section{Basic properties}\label{sec2}

\begin{lem}\label{baslem1}
For all $a,n\in\N_{0}$ we have the identities
\begin{align*}
&B_{2^{a}n-1}(t)=\frac{t^a-1}{t-1}B_{n}(t)+B_{n-1}(t),\\
&B_{2^{a}n+1}(t)=\frac{t^a-1}{t-1}B_{n}(t)+B_{n+1}(t).
\end{align*}
\end{lem}
\begin{proof}
We will proceed by induction on $a$ in order to prove the first
equality. The result is true for $a=0$ and $a=1$. Suppose that the
statement is true for $a$ and all $n\geq 1$. We have
\begin{align*}
B_{2^{a+1}n-1}(t)&=B_{2^{a}n}(t)+B_{2^{a}n-1}(t)=t^{a}B_{n}(t)+\frac{t^a-1}{t-1}B_{n}(t)+B_{n-1}(t)\\
                 &=\frac{t^{a+1}-1}{t-1}B_{n}(t)+B_{n-1}(t),
\end{align*}
and the first equality is proved.

Because the proof of the second equality goes in exactly the same
manner we leave it to the reader.
\end{proof}


\begin{cor}\label{+-expression}
For each $n\in\N$ we have
\begin{equation*}
B_{2^{n}-1}(t)=\frac{t^{n}-1}{t-1},\quad B_{2^{n}}(t)=t^{n},\quad
B_{2^{n}+1}(t)=\frac{t^n-1}{t-1}+t.
\end{equation*}
\end{cor}

One of the main properties of the Stern sequence is the symmetry
property: $s(2^{k}+i)=s(2^{k+1}-i)$ for $i=0,1,\ldots, 2^{k}$. It is
easy to see that the sequence of Stern polynomials do not satisfy
any identity of the form $B_{2^{k}+i}(t)-B_{2^{k+1}-i}(t)=f(t)$
where $f$ is a polynomial which is independent of $k$ and $i$.
Indeed, we have $B_{2^2+1}(t)-B_{2^3-1}(t)=-t(t-1)$ and
$B_{2^{2}+2^{2}-1}(t)-B_{2^3-2^2+1}(t)=t(t-1)$. However, these
identities and the others examined with the use of computer lead us
to conjecture the following theorem.

\begin{thm}\label{symprop}
The sequence of Stern polynomials satisfy the following symmetry
property:
\begin{equation*}
B_{2^{n+1}-i}(t)-B_{2^{n}+i}(t)=\begin{cases}
\begin{array}{ll}
  t(t-1)B_{2^{n-1}-i}(t)  &\mbox{for}\;i=0,1,\ldots,2^{n-1}, \\
  -t(t-1)B_{i-2^{n-1}}(t) &\mbox{for}\;i=2^{n-1}+1,\ldots, 2^{n}.
\end{array}
\end{cases}
\end{equation*}
\end{thm}
\begin{proof}
We start with the first equality. We proceed by induction on $n$ and
$0\leq i\leq 2^{n-1}$. The equality is true for $n=1$ and $i=0,1$.
Let us assume that it holds for $n$ and $0\leq i\leq 2^{n-1}.$ We
prove it for $n+1$ and $0\leq i\leq 2^{n}.$

If $i=2m$ then $0\leq m\leq 2^{n-1}$ and we have the sequence of
equalities
\begin{align*}
B_{2^{n+2}-i}(t)-B_{2^{n+1}+i}(t)&=B_{2^{n+2}-2m}(t)-B_{2^{n+1}+2m}(t)\\
                                 &=t(B_{2^{n+1}-m}(t)-B_{2^{n}+m}(t))\\
                                 &=t^2(t-1)B_{2^{n-1}-m}(t)=t(t-1)B_{2^{n}-2m}\\
                                 &=t(t-1)B_{2^{n}-i}(t).
\end{align*}

If $i=2m-1$ then $1\leq m\leq 2^{n-1}$ and we have the sequence of
equalities
\begin{align*}
B_{2^{n+2}-i}(t)-B_{2^{n+1}+i}(t)&=B_{2^{n+2}-2m+1}(t)-B_{2^{n+1}+2m-1}(t)\\
                                 &=B_{2(2^{n+1}-m)+1}(t)-B_{2(2^{n}+m-1)+1}(t)\\
                                 &=B_{2^{n+1}-m}(t)+B_{2^{n+1}-m+1}(t)-B_{2^{n}+m-1}(t)-B_{2^{n}+m}(t)\\
                                 &=t(t-1)B_{2^{n-1}-m}(t)-t(t-1)B_{2^{n-1}-(m-1)}(t)\\
                                 &=t(t-1)B_{2(2^{n-1}-m)+1}(t)=t(t-1)B_{2^{n}-(2m-1)}(t).
\end{align*}

The second equality can be proved in an analogous manner, so we left
this computation to the reader.
\end{proof}

\begin{thm}
Let $\mu(n)$ be the highest power of $2$ dividing $n$. Then the
following identity holds
\begin{equation*}
t^{\mu(n)}(B_{n+1}(t)+B_{n-1}(t))=(B_{2^{\mu(n)}+1}(t)+B_{2^{\mu(n)}-1}(t))B_{n}(t).
\end{equation*}
In particular, if $n$ is odd, we have that
\begin{equation*}
B_{n+1}(t)+B_{n-1}(t)=tB_{n}(t).
\end{equation*}
\end{thm}
\begin{proof}
First we consider the case $n$-odd. Then $n=2m+1$ for some $m\in\N$
and we have that $\mu(n)=0$. Now we find that
\begin{equation*}
B_{n+1}(t)+B_{n-1}(t)=B_{2m+2}(t)+B_{2m}(t)=tB_{m+1}(t)+tB_{m}(t)=tB_{2m+1}(t)=tB_{n}(t).
\end{equation*}
and our theorem follows in case of odd $n$.

If $n$ is even then $n=2^{\mu(n)}(2m+1)$ for some $m\in\N$, and to
shorten the notation let us put $\mu=\mu(n)$. We compute
\begin{align*}
t^{\mu}(B_{n+1}(t)&+B_{n-1}(t))\\
                  &=t^{\mu}(B_{2^{\mu}(2m+1)+1}(t)+B_{2^{\mu}(2m+1)-1}(t))\\
                              &=t^{\mu}\left(\frac{t^{\mu}-1}{t-1}B_{2m+1}(t)+B_{2m+2}(t)+\frac{t^{\mu}-1}{t-1}B_{2m+1}(t)+B_{2m}(t)\right)\\
                              &=t^{\mu}\left(2\frac{t^{\mu}-1}{t-1}B_{2m+1}(t)+t(B_{m+1}(t)+B_{m}(t))\right)\\
                              &=t^{\mu}\left(2\frac{t^{\mu}-1}{t-1}B_{2m+1}(t)+tB_{2m+1}(t)\right)\\
                              &=\left(2\frac{t^{\mu}-1}{t-1}+t\right)t^{\mu}B_{2m+1}(t)=\left(2\frac{t^{\mu}-1}{t-1}+t\right)B_{2^{\mu}(2m+1)}\\
                              &=(B_{2^{\mu(n)}+1}(t)+B_{2^{\mu(n)}-1}(t))B_{n}(t),
\end{align*}
and the theorem follows.
\end{proof}

The next interesting property of the Stern polynomials is contained
in the following.

\begin{thm}\label{baslem3}
For $0\leq k\leq 2^{n}-2$ we have
\begin{equation*}
B_{k+1}(t)B_{2^{n}-k}(t)-B_{k}(t)B_{2^{n}-k-1}(t)=t^{n}.
\end{equation*}
\end{thm}
\begin{proof}
We proceed by induction on $n$ and $0\leq k\leq 2^{n}-2$. The
identity is true for $n=1, k=0$ and $n=2,k=0,1,2$. Let us suppose
that our identity holds for given $n$ and $0\leq k\leq 2^{n}-2$. We
will prove that the identity holds for $n+1$.

If $0\leq k\leq 2^{n+1}-2$ and $k$ is even then we have $k=2i$ and
$0\leq i\leq 2^{n}-1$. If $i=2^{n}-1$ then it is easy to show that
our identity holds, so we can assume that $i\leq 2^{n}-2$. Then we
have
\begin{align*}
B_{k+1}(t)&B_{2^{n+1}-k}(t)-B_{k}(t)B_{2^{n+1}-k-1}(t)\\
          &=B_{2i+1}(t)B_{2^{n+1}-2i}(t)-B_{2i}(t)B_{2^{n+1}-2i-1}(t)\\
          &=tB_{i}(t)B_{2^{n}-i}(t)+tB_{i+1}(t)B_{2^{n}-i}(t)-tB_{i}(t)B_{2^{n}-i-1}(t)-tB_{i}(t)B_{2^{n}-i}(t)\\
          &=t(B_{i+1}(t)B_{2^{n}-i}(t)-B_{i}(t)B_{2^{n}-i-1}(t))=t^{n+1},
\end{align*}
where the last equality follows from the induction hypothesis.

If $0\leq k\leq 2^{n+1}-2$ and $k$ is odd then we have $k=2i+1$ and
$0\leq i\leq 2^{n}-2$. We have
\begin{align*}
B_{k+1}(t)&B_{2^{n+1}-k}(t)-B_{k}(t)B_{2^{n+1}-k-1}(t)\\
          &=B_{2i+2}(t)B_{2^{n+1}-2i-1}(t)-B_{2i+1}(t)B_{2^{n+1}-2i-2}(t)\\
          &=tB_{i+1}(t)(B_{2^{n}-i-1}(t)+B_{2^n-i}(t))-tB_{2^{n}-i-1}(t)(B_{i}(t)+B_{i+1}(t))\\
          &=t(B_{i+1}(t)B_{2^{n}-i}(t)-B_{i}(t)B_{2^{n}-i-1}(t))=t^{n+1},
\end{align*}
and the theorem follows.
\end{proof}

As an immediate consequence of Theorem \ref{baslem3} we get the
following result.

\begin{cor}\label{gcd}
\begin{enumerate}
\item For each $n\in\N$ we have $\gcd(B_{n}(t),B_{n+1}(t))=1.$

\item
If $a,b\in\N$ are odd and $a+b=2^{n}$ for some $n$ then
$\gcd(B_{a}(t),B_{b}(t))=1$.

\end{enumerate}
\end{cor}
\begin{proof}
(1) From  Theorem \ref{baslem3} we deduce that if a polynomial
$h\in\Z[t]$ divides $\gcd(B_{n}(t),B_{n+1}(t))$ for some $n$ then
$h(t)=t^{m}$ for some $0\leq m\leq n$. If $m\geq 1$ we get that
$B_{n}(0)=B_{n+1}(0)=0$, but from the Theorem \ref{specialvalues}
(which will be proved later) we know that for odd $k$ we have
$B_{k}(0)=1$. This implies that $m=0$ and the result follows.

(2) If $a+b=2^{n}$ then $b=2^{n}-a$ and using similar reasoning as
in the previous case we deduce that if
$h(t)|\gcd(B_{a}(t),B_{2^{n}-a}(t))$ then $h(t)=t^{m}$ for some $m$.
But $a$ is odd, thus $B_{a}(0)=1$ and we get that $m=0$ and
$h(t)=1$.
\end{proof}

Now we give some extremal properties of the sequence of Stern
polynomials. More precisely, for given positive real number $a$ we
ask what is the maximum (minimum) of $B_{i}(a)$ for $i\in [2^{n-1},
2^{n}]$. We prove the following theorem.

\begin{thm}\label{extremalofB}
\begin{enumerate}
\item Let $a$ be a real number satisfying $a>2$. Then we have
\begin{equation*}
  M_{n}(a)=\op{max}\{B_{i}(a):\;i\in [2^{n-1},2^{n}]\}=a^{n}=B_{2^{n}}(a).
\end{equation*}

\item Let $a\in (0,2)$. Then we have
\begin{equation*}
m_{n}(a)=\op{min}\{B_{i}(a):\;i\in [2^{n-1},2^{n}]\}=
\begin{cases}
\begin{array}{ll}
  a^{n} & \mbox{for}\quad a\in (0,1], \\
  a^{n-1} & \mbox{for}\quad a\in (1,2].\\
\end{array}
\end{cases}
\end{equation*}

\end{enumerate}
\end{thm}
\begin{proof}
(1) In order to prove the identity for $M_{n}(a)$ we proceed by
induction on $n$. For $n=1$ we have $M_{1}(a)=\op{max}\{1,a\}=a$.
Similarly for $n=2$ we have
$M_{2}(a)=\op{max}\{a,a+1,a^{2}\}=a^{2}$. Thus our theorem is true
for $n=1,2.$ Let us suppose that $M_{n}(a)=a^{n}$. We will show that
$M_{n+1}(a)=aM_{n}(a)$. We have:
\begin{align*}
M_{n+1}&(a)=\op{max}\{B_{i}(a):\;i\in [2^{n},2^{n+1}]\}\\
       &=\op{max}\{\op{max}\{B_{2i}(a):\;i\in [2^{n-1},2^{n}]\}, \op{max}\{B_{2i+1}(a):\;i\in [2^{n-1}, 2^{n}-1]\}\}\\
       &=\op{max}\{a\op{max}\{B_{i}(a):\;i\in [2^{n-1},2^{n}]\}, \op{max}\{B_{2i+1}(a):\;i\in [2^{n-1}, 2^{n}-1]\}\}.
\end{align*}
Now from the induction hypothesis we have
\begin{align*}
\op{max}&\{B_{2i+1}(a):\;i\in [2^{n-1},2^{n}-1]\}\\
        &=\op{max}\{B_{i+1}(a)+B_{i}(a):\;i\in [2^{n-1},2^{n}-1]\}<2M_{n}(a).
\end{align*}
Because $a>2$ we get that $M_{n+1}(a)=aM_{n}(a)$ and our theorem
follows.

\bigskip
(2) In order to prove the second identity we consider only the case
of $a\in(1,2]$ because the case of $a\in(0,1]$ is completely
analogous. We proceed by induction on $n$. We take $a\in(1,2]$ and
note that for $n=1$ we have $m_{1}(a)=\op{min}\{1,a\}=1$. Similarly,
for $n=2$ we have $m_{2}(a)=\op{min}\{a,a+1,a^{2}\}=a$. Thus our
theorem is true for $n=1,2.$ Let us suppose that $m_{n}(a)=a^{n-1}$
for a given $n$. We will show that $m_{n+1}(a)=am_{n}(a)$. We have:
\begin{align*}
m_{n+1}(a)&=\op{min}\{B_{i}(a):\;i\in [2^{n},2^{n+1}]\}\\
          &=\op{min}\{\op{min}\{B_{2i}(a):\;i\in [2^{n-1},2^{n}]\}, \op{min}\{B_{2i+1}(a):\;i\in [2^{n-1}, 2^{n}-1]\}\}\\
          &=\op{max}\{a\op{min}\{B_{i}(a):\;i\in [2^{n-1},2^{n}]\}, \op{min}\{B_{2i+1}(a):\;i\in [2^{n-1},
          2^{n}-1]\}\}.
\end{align*}
Now from the induction hypothesis we have
\begin{align*}
\op{min}&\{B_{2i+1}(a):\;i\in[2^{n-1},2^{n}-1]\}\\
        &=\op{min}\{B_{i+1}(a)+B_{i}(a):\;i\in [2^{n-1},2^{n}-1]\}>2m_{n}(a).
\end{align*}
Because $a\in(1,2]$ we get that $m_{n+1}(a)=am_{n}(a)$ and our
theorem follows.
\end{proof}

From the above theorem we easily deduce the following.

\begin{cor}\label{ineqforB}
For $a\in(0,1]$ we have that $B_{n}(a)\geq a^{\lg n}=n^{\lg a}$. For
$a\in(1,2]$ we have $B_{n}(a)\geq \frac{1}{a}n^{\lg a}$. Finally,
for $a>2$ we have an inequality $B_{n}(a)\leq n^{\lg a},$ where
$\lg$ stands for logarithm in base 2.
\end{cor}

\section{Generating function and its consequences}\label{sec3}

In this section we give a closed formula for the ordinary generating
function of the sequence of Stern polynomials and then use its
properties in order to obtain several interesting identities
satisfied by Stern polynomials. So let us define
\begin{equation*}
B(t,x)=\sum_{n=0}^{\infty}B_{n}(t)x^n.
\end{equation*}

Now using the recurrence relations satisfied by the polynomials
$B_{n}(t)$ we can write
\begin{align*}
B(t,x)&=\sum_{n=0}^{\infty}B_{2n}(t)x^{2n}+\sum_{n=0}^{\infty}B_{2n+1}(t)x^{2n+1}\\
      &=t\sum_{n=0}^{\infty}B_{n}(t)x^{2n}+x\sum_{n=0}^{\infty}B_{n}(t)x^{2n}+\sum_{n=0}^{\infty}B_{n+1}(t)x^{2n+1}\\
      &=\left(t+x+\frac{1}{x}\right)B(t,x^2).
\end{align*}
From the above computation we get that the function $B(t,x)$
satisfies the functional equation
\begin{equation}\label{funceq}
(1+tx+x^2)B(t,x^2)=xB(t,x).
\end{equation}

We prove the following theorem.

\begin{thm}\label{genfunt}
The sequence of Stern polynomials has the generating function
\begin{equation}\label{genfun}
B(t,x)=x\prod_{n=0}^{\infty}(1+tx^{2^n}+x^{2^{n+1}}).
\end{equation}
\end{thm}

One can easily check that the function defined by the right hand
side of (\ref{genfun}) satisfies the equation (\ref{funceq}).
However it is not clear that this is the only solution of
(\ref{funceq}). In order to prove this we will need some result
concerning the polynomials defined by product
\begin{equation*}
F_{n}(t,x)=x\prod_{i=0}^{n}(1+tx^{2^i}+x^{2^{i+1}}).
\end{equation*}

Our approach is similar to the one used in the paper \cite{DilSto}
where another generalization of the Stern diatomic sequence is
considered. We prove the following expansion.

\begin{thm}\label{exprodt}
For any $n\in\N$ we have
\begin{equation}\label{exprod}
F_{n}(t,x)=\sum_{i=1}^{2^{n+1}}(B_{i}(t)+B_{2^{n+1}-i}(t)x^{2^{n+1}})x^{i}.
\end{equation}

\end{thm}
\begin{proof}
We proceed by induction on $n$. For $n=0$ we have by definition,
\begin{equation*}
F_{0}(t,x)=x+tx^2+x^3=(B_{1}(t)+B_{1}(t)x^2)x+tx^2.
\end{equation*}
Now suppose that (\ref{exprod}) holds for $n$, i.e., we have
\begin{equation}\label{parass}
F_{n}(t,x)=\sum_{i=1}^{2^{n+1}}(B_{i}(t)+B_{2^{n+1}-i}(t)x^{2^{n+1}})x^{i}=:f_{n}(t,x).
\end{equation}
Multiplying  both sides of this identity by
$1+tx^{2^{n+1}}+x^{2^{n+2}}$ we see that in order to get the
statement it is enough to show that $f_{n+1}(t,x)=F_{n+1}(t,x).$ We
find that:
\begin{align*}
f&_{n+1}(t,x)=\sum_{i=1}^{2^{n+2}}B_{i}(t)x^{i}+\sum_{i=1}^{2^{n+2}}B_{2^{n+2}-i}(t)x^{2^{n+2}+i}+\sum_{i=1}^{2^{n+1}}B_{2i}(t)x^{2i}\\
&+\sum_{i=1}^{2^{n+1}}B_{2i-1}(t)x^{2i-1}+\sum_{i=1}^{2^{n+1}}B_{2^{n+2}-2i}(t)x^{2^{n+2}+2i}+\sum_{i=1}^{2^{n+1}}B_{2^{n+2}-2i+1}(t)x^{2^{n+2}+2i-1}\\
            &=t\sum_{i=1}^{2^{n+1}}(B_{i}(t)+B_{2^{n+1}-i}(t)x^{2^{n+2}})x^{2i}+\sum_{i=1}^{2^{n+1}}(B_{i}(t)+B_{2^{n+1}-i}(t)x^{2^{n+2}})x^{2i-1}\\
            &+\sum_{i=1}^{2^{n+1}}(B_{i-1}(t)+B_{2^{n+1}-i+1}(t)x^{2^{n+2}})x^{2i-1}.\\
         \end{align*}
Now let us note that the first term on the left hand side of the
last equality is equal to $tF_{n}(t,x^2)$, the second is equal to
$F_{n}(t,x^2)$. Finally, substituting $i=j+1$ in the third term we
get
\begin{align*}
\sum_{i=1}^{2^{n+1}}&(B_{i-1}(t)+B_{2^{n+1}-i+1}(t)x^{2^{n+2}})x^{2i-1}\\
                    &=B_{0}(t)+B_{2^{n+1}}(t)x^{2^{n+2}+1}-(B_{2^{n+1}}(t)+B_{0}(t)x^{2^{n+2}})x^{2^{n+2}+1}\\
                    &+x\sum_{j=1}^{2^{n+1}}(B_{j}(t)+B_{2^{n+1}-j}(t)x^{2^{n+2}})x^{2j}\\
                    &=x\sum_{j=1}^{2^{n+1}}(B_{j}(t)+B_{2^{n+1}-j}(t)x^{2^{n+2}})x^{2j}=xF_{n}(t,x^2).
\end{align*}
Our reasoning shows that
\begin{equation*}
f_{n+1}(t,x)=\frac{1}{x}(1+tx+x^2)F_{n}(t,x^2)=F_{n+1}(t,x),
\end{equation*}
and theorem follows.
\end{proof}

Our first application of the above theorem will be a proof of
Theorem \ref{genfunt}.

\begin{proof}[Proof of the Theorem \ref{genfunt}] First of all let us
note that the generating series for the sequence of Stern
polynomials is convergent for any fixed $|x|<1$. This is an easy
consequence of the inequality $|B_{n}(t)|\leq n$ for $|t|\leq 2$
which follows from the inequality $|B_{n}(t)|\leq B_{n}(2)=n$ and
$|B_{n}(t)|\leq n^{\lg|t|}$ which holds for $|t|>2$ and follows from
the Corollary \ref{ineqforB}. On the other hand we know from the
theory of infinite products that the product
$\prod_{n=0}^{\infty}(1+tx^{2^n}+x^{2^{n+1}})$ is convergent for a
given $t\in\R$ and any $x$ satisfying inequality $|x|<1$. Now in
order to prove the identity (\ref{genfun}), in view of
(\ref{exprod}), it is enough to show that the sum
\begin{equation*}
R_{N}(t,x):=\sum_{i=1}^{2^{N+1}}B_{2^{N+1}-i}(t)x^{2^{N+1}+i},
\end{equation*}
converges to $0$ as $N\rightarrow\infty$, for $|x|<1$ and $t\in\R$.
Using now the estimate $|B_{2^{N+1}-i}(t)|\leq (2^{N+1})^{\lg|t|},$
we get
\begin{equation*}
|R_{N}(t,x)|<|x|^{2^{N+1}}2^{(N+1)\lg
|t|}\sum_{i=1}^{2^{N+1}}|x|^{i}<|x|^{2^{N+1}}2^{(N+1)\lg
|t|}\frac{|x|}{1-|x|}.
\end{equation*}
It is clear that under our assumption concerning $x$ the left hand
side of the above inequality tends to $0$ with $N\rightarrow\infty.$
\end{proof}

One of the many interesting properties of the Stern diatomic
sequence $s(n)=B_{n}(1)$ is the existence of a closed formula for
the sum of all elements from the first $k$ rows of the diatomic
array. More precisely we have:
$\sum_{i=1}^{2^{k}}s(n)=\frac{3^{k}+1}{2}.$ It is a natural question
if a similar result can be obtained for the sequence of Stern
polynomials. As we will see such a generalization can be obtained
with the help of the expansion (\ref{exprod}). More precisely we
have the following.

\begin{cor}\label{sumrows}
For any $k\geq 0$ we have
\begin{equation*}
\sum_{i=1}^{2^{k}}B_{i}(t)=\frac{1}{2}((t+2)^{k}+t^{k}).
\end{equation*}
\end{cor}
\begin{proof}
In order to prove this we set $x=1$ and $n=k-1$ in the expansion
(\ref{exprod}) and get
\begin{equation*}
(t+2)^k=\sum_{i=0}^{2^{k}}(B_{i}(t)+B_{2^{k}-i}(t))-B_{2^{k}}(t)=2\sum_{i=1}^{2^{k}}B_{i}(t)(t)-t^{k},
\end{equation*}
and the result follows.
\end{proof}

A simple application of Corollary \ref{sumrows} leads to the
following.

\begin{cor}\label{altsumrows}
For any $k\geq 1$ we have
\begin{equation*}
\sum_{i=1}^{2^{k}}(-1)^{i}B_{i}(t)=\frac{t-2}{2}((t+2)^{k-1}+t^{k-1}))+t^{k-1}.
\end{equation*}
\end{cor}

The following theorem summarizes some elementary manipulations of
the generating function for the sequence of Stern polynomials.
\begin{thm}\label{genfunrel} Let $B(t,x)$ be a generating function for
the sequence for Stern polynomials. Then we have:
\begin{equation}\label{genidentity1}
(1+tx+x^2)B(t,x^2)=xB(t,x),
\end{equation}


\begin{equation}\label{genidentity3}B(-t,x)B(t,x)=B(2-t^2,x^2),\end{equation}

\begin{equation}\label{genidentity4}
B\left(-t^2-\frac{1}{t^2},x^2\right)=\frac{1}{\left(1-\frac{1}{t^2}x\right)(1-t^2x)}B\left(-t^2-\frac{1}{t^2},x\right).
\end{equation}
\end{thm}
\begin{proof}
The first three displayed identities are easy consequences of the
manipulation of the generating function for Stern polynomials. The
fourth identity follows from the first identity and the fact that
\begin{equation*}
B\left(t+\frac{1}{t},x\right)B\left(-t-\frac{1}{t},x\right)=B\left(-t^2-\frac{1}{t^2},x^2\right).
\end{equation*}
\end{proof}

We use Theorem \ref{genfunrel} to get several interesting identities
related to the sequence of Stern polynomials. However, before we do
that we recall an useful power series expansion. We have
\begin{equation*}
\frac{1}{1-2tx+x^2}=\sum_{n=0}^{\infty}U_{n}(t)x^{n},
\end{equation*}
where $U_{n}(t)$ is the Chebyshev polynomial of the second kind.



Now, we are ready to prove the following theorem.
\begin{thm}\label{identities}
The following identities holds:
\begin{equation}\label{identity1}
\sum_{i=0}^{n}B_{i}(t)U_{n-i}(-\frac{t}{2})=\begin{cases}\begin{array}{ll}
                                                0,                   & \mbox{if}\quad 2|n  \\
                                                 B_{\frac{n+1}{2}}(t), & \mbox{otherwise}
                                               \end{array}
\end{cases}
\end{equation}


\begin{equation}\label{identity3}
\sum_{i=0}^{n}B_{i}(t)B_{n-i}(-t)=\begin{cases}\begin{array}{ll}
                                                 B_{\frac{n}{2}}(2-t^2), & \mbox{if}\quad 2|n  \\
                                                 0, & \mbox{otherwise}
                                               \end{array}
\end{cases}
\end{equation}

\begin{equation}\label{identity4}
\frac{t^2}{t^4-1}\sum_{i=0}^{n}\left(t^{2i}-\frac{1}{t^{2i}}\right)B_{n-i}\left(-t^2-\frac{1}{t^2}\right)=\begin{cases}\begin{array}{ll}
                                                 B_{\frac{n}{2}}\left(-t^2-\frac{1}{t^2}\right), & \mbox{if}\quad 2|n  \\
                                                 0, & \mbox{otherwise}
                                               \end{array}
\end{cases}
\end{equation}
\end{thm}
\begin{proof}
Before we prove our theorem let us recall that if
$A(x)=\sum_{n=0}^{\infty}a_{n}x^{n}$ and
$B(x)=\sum_{n=0}^{\infty}b_{n}x^{n}$ then
$A(x)B(x)=\sum_{n=0}^{\infty}c_{n}x^{n}$ where
$c_{n}=\sum_{i=0}^{n}a_{i}b_{n-i}$.

Now it is easy to see that the identity (\ref{identity1}) follows
from the identity (\ref{genidentity1}) given in Theorem
\ref{genfunrel}.
Similarly from the identity (\ref{genidentity3}) we get
(\ref{identity3}). Finally, in order to get the last identity we
note that
\begin{equation*}
\frac{1}{(1-t^{2}x)(1-\frac{1}{t^2}x)}=\frac{t^{4}}{t^4-1}\left(\frac{t^{4}}{1-t^{2}x}-\frac{1}{1-\frac{1}{t^2}x}\right)=\frac{t^2}{t^4-1}\sum_{n=0}^{\infty}\left(t^{2n}-\frac{1}{t^{2n}}\right)x^{n}
\end{equation*}
and use the identity (\ref{genidentity4}).
\end{proof}

Another interesting property of the sequence of Stern polynomials is
contained in the following.

\begin{thm}
Let $\nu(n)$ denote the number of 1's in the unique binary
representation of $n$. Then we have the identity
\begin{equation*}
B_{n+1}\left(t+\frac{1}{t}\right)=\sum_{i=0}^{n}t^{\nu(n-i)-\nu(i)}.
\end{equation*}
\end{thm}
\begin{proof}
In order to prove the identity let us recall that if $\nu(n)$
denotes number of 1's in the unique binary representation of $n$
then we have an identity \cite{Rez}
\begin{equation*}
\prod_{n=0}^{\infty}(1+tx^{2^n})=\sum_{n=0}^{\infty}t^{\nu(n)}x^{n}.
\end{equation*}
Now let us note that
\begin{align*}
B\left(t+\frac{1}{t},x\right)&=\sum_{n=0}^{\infty}B_{n+1}\left(t+\frac{1}{t}\right)x^{n+1}=x\prod_{n=0}^{\infty}(1+tx^{2^{n}})\prod_{n=0}^{\infty}\left(1+\frac{1}{t}x^{2^{n}}\right)\\
                             &=\sum_{n=0}^{\infty}\left(\sum_{i=0}^{n}t^{\nu(n-i)-\nu(i)}\right)x^{n+1},
\end{align*}
and we get the desired identity.
\end{proof}

\begin{cor}\label{corl}
Let $\nu(n)$ denote the number of 1's in the unique binary
representation of $n$ and let $e(n)=\op{deg}B_{n}(t)$. Then we have:
\begin{align*}
e(n+1)=&\op{max}\{\nu(n-i)-\nu(i):\;i=0,\ldots,n\},\\
-e(n+1)=&\op{min}\{\nu(n-i)-\nu(i):\;i=0,\ldots,n\}.
\end{align*}
\end{cor}

\section{Properties of the sequence
$e(n)=\op{deg}_{t}B_{n}(t)$}\label{sec4}

Let $e(n)=\op{deg}_{t}B_{n}(t)$. From the definition of the
$B_{n}(t)$ it is easy to see that the sequence $e(n)$ satisfies the
following relations:
\begin{equation*}
e(1)=0,\; e(2n)=e(n)+1,\;e(2n+1)=\op{max}\{e(n),e(n+1)\}.
\end{equation*}
The sequence thus starts as
\begin{equation*}
0, 1, 1, 2, 1, 2, 2, 3, 2, 2, 2, 3, 2, 3, 3, 4, 3, 3, 2, 3, 2, 3, 3,
4, 3 , 3, 3, 4, 3, 4, 4, 5, \ldots .
\end{equation*}

In \cite[Corollary 13]{Kla} it was shown that the sequence $e(n)$
can be defined alternatively as follows:
\begin{equation*}
e(1)=0,\; e(2n)=e(n)+1,\; e(4n+1)=e(n)+1,\; e(4n+3)=e(n+1)+1.
\end{equation*}
This is a useful definition and we will use it many times in the
following sections.

We start with the problem of counting the number of Stern
polynomials with degree equal to $n$.

\begin{thm}\label{genfundeg}
We have an identity
\begin{equation*}
E(x)=\sum_{n=1}^{\infty}x^{e(n)}=\frac{1}{1-3x}.
\end{equation*}
\end{thm}
\begin{proof}
It is clear that
$E(x)=\sum_{n=1}^{\infty}x^{e(n)}=\sum_{n=1}^{\infty}C_{n}x^{n},$
where $C_{n}=|\{i\in\N_{+}:\;e(i)=n\}|$. In order to prove our
theorem we note that
\begin{align*}
E(x)&=\sum_{n=1}^{\infty}x^{e(n)}=\sum_{n=1}^{\infty}x^{e(2n)}+\sum_{n=1}^{\infty}x^{e(4n-1)}+\sum_{n=1}^{\infty}x^{e(4n-3)}\\
    &=\sum_{n=1}^{\infty}x^{e(n)+1}+\sum_{n=1}^{\infty}x^{e(n)+1}+1+\sum_{n=2}^{\infty}x^{e(n-1)+1}\\
    &=xE(x)+xE(x)+1+xE(x)=3xE(x)+1.
\end{align*}
Solving the above (linear) equation for $E$ we get the expression
displayed in the statement of theorem.
\end{proof}

From the above theorem we deduce the following.
\begin{cor}
Let $C_{n}=|\{i\in\N_{+}:\;e(i)=n\}|$. Then $C_{n}=3^{n}$.
\end{cor}

Further properties of the sequence $\{e(n)\}_{n\in\N_{+}}$ are
contained in the following theorem.

\begin{thm}\label{extremalofe} We have the following equalities:
\begin{equation}\label{eidentity1}
m(n):=\op{min}\{e(i):\;i\in
[2^{n-1},2^{n}]\}=\left\lfloor\frac{n}{2}\right\rfloor,\;n\geq 2,
\end{equation}
\begin{equation}\label{eidentity2}
M(n):=\op{max}\{e(i):\;i\in [2^{n-1},2^{n}]\}=n,
\end{equation}
\begin{equation}\label{eidentity3}
\op{mdeg}(n):=\op{min}\{i:\;e(i)=n\}=2^{n},
\end{equation}
\begin{equation}\label{eidentity4}
\op{Mdeg}(n):=\op{max}\{i:\;e(i)=n\}=\frac{4^{n+1}-1}{3}.
\end{equation}
\end{thm}
\begin{proof}
In order to prove (\ref{eidentity1}) first we show that $m(n)\geq
\left\lfloor\frac{n}{2}\right\rfloor$. It is clear that it is enough
to show that $e(2^{n-1}+i)\geq \left\lfloor\frac{n}{2}\right\rfloor$
for $i\in[0,2^{n-1}]$. We proceed by induction on $n$ and
$i\in[0,2^{n-1}]$. This inequality is true for $n=2$ and $i=0,1,2$.
Let us suppose that it is true for $n$ and $i\in[0,2^{n-1}]$. We
show that it is true for for $n+1$ and $i\in[0,2^{n}]$. If $i$ is
even then $i=2j$ for some $j\in[0,2^{n-1}]$ and we have
\begin{equation*}
e(2^{n}+i)=e(2^{n}+2j)=e(2^{n-1}+j)+1\geq
\left\lfloor\frac{n}{2}\right\rfloor+1\geq
\left\lfloor\frac{n+1}{2}\right\rfloor.
\end{equation*}
If $i$ is odd then $i=4j+1$ or $i=4j+3$. In the case of $i=4j+1$ we
have $j\in[0,2^{n-2}-1]$ and
\begin{equation*}
e(2^{n}+i)=e(2^{n}+4j+1)=e(2^{n-2}+j)+1\geq
\left\lfloor\frac{n-1}{2}\right\rfloor+1=
\left\lfloor\frac{n+1}{2}\right\rfloor,
\end{equation*}
In the case of $i=4j+3$ we have $j\in[0,2^{n-2}-3]$ and
\begin{equation*}
e(2^{n}+i)=e(2^{n}+4j+3)=e(2^{n-2}+j+1)+1\geq
\left\lfloor\frac{n-1}{2}\right\rfloor+1=
\left\lfloor\frac{n+1}{2}\right\rfloor.
\end{equation*}
This finishes proof of the inequality $m(n)\geq
\left\lfloor\frac{n}{2}\right\rfloor.$ In order to show that for
each $n\geq 2$ this inequality is strict it is enough to give an
integer $a_{n}$ such that $a_{n}\in[2^{n-1},2^{n}]$ and
$e(a_{n})=\left\lfloor\frac{n}{2}\right\rfloor.$ Let us define
\begin{equation*}
a_{2k}=\frac{1}{3}(2^{2k+1}+1), \quad
a_{2k+1}=2^{2k-1}+\frac{1}{3}(2^{2k+1}+1).
\end{equation*}
First of all let us note that
\begin{equation*}
2^{2k-1}<a_{2k}<2^{2k},\quad 2^{2k}<a_{2k+1}<2^{2k+1}
\end{equation*}
for each $k\geq 1$. It is easy to check that
\begin{equation}\label{akrel}
a_{2(k+1)}=4a_{2k}-1,\quad a_{2(k+1)+1}=4a_{2k+1}-1.
\end{equation}
In order to finish the proof we check that
$e(a_{2k})=e(a_{2k+1})=k.$ We proceed by induction on $k$. For $k=1$
we have $a_{2}=3, a_{3}=5$ and clearly $e(3)=e(5)=1$. Let us suppose
that for $k$ we have $e(a_{2k})=e(a_{2k+1})=k.$ We will prove the
equality $e(a_{2(k+1)})=e(a_{2(k+1)+1})=k+1.$ Using now the first
relation from (\ref{akrel}) we get
\begin{equation*}
e(a_{2(k+1)})=e(4(a_{2k}-1)+3)=e(2(a_{2k}-1)+2)=e(a_{2k})+1=k+1.
\end{equation*}
Using now the second relation from (\ref{akrel}) we get
\begin{equation*}
e(a_{2(k+1)+1})=e(4(a_{2k+1}-1)+3)=e(2(a_{2k+1}-1)+2)=e(a_{2k+1})+1=k+1
\end{equation*}
and the result follows.

 Similarly as in the case of computation of $m(n)$ we see that in
order to prove the formula for $M(n)$ it is enough to show that
there exists an integer $k\in[2^{n-1},2^{n}]$ such that $e(k)=n$ and
that $e(2^{n-1}+i)\leq n$ for $i\in[0,2^{n-1}]$. We have that
$e(2^n)=n$, so we are left with showing the inequality
$e(2^{n-1}+i)\leq n$ for $i\in[0,2^{n-1}]$. We proceed by induction
on $n$ and $i\in[0,2^{n-1}]$. This inequality is true for $n=1$ and
$i=0,1$. Let us suppose that it is true for $n$ and
$i\in[0,2^{n-1}]$. We show that it is true for $n+1$ and
$i\in[0,2^{n}]$. If $i$ is even then $i=2m$ for some
$m\in[0,2^{n-1}]$ and we have
\begin{equation*}
e(2^{n+1}+i)=e(2^{n+1}+2m)=e(2^{n}+m)+1\leq n+1.
\end{equation*}
If $i$ is odd then $i=2m+1$ for  some $m\in[0,2^{n-1}]$ (note that
$e(2^{n}+1)=n$ which is a consequence of Corollary
\ref{+-expression}). Now we have
\begin{equation*}
e(2^{n+1}+i)=e(2^{n+1}+2m+1)=\op{max}\{e(2^{n}+m),
e(2^{n}+m+1)\}\leq n+1,
\end{equation*}
and the equality $M(n)=n$ follows.

In order to prove (\ref{eidentity3}) we show that $e(2^{n})=n$ and
$e(2^{n}-i)<n$ for $i=1,2,\ldots,2^{n}$. The equality follows form
the identity $B_{2^{n}}(t)=t^{n}$. In order to prove the inequality
$e(2^{n}-i)<n$ for $i=1,2,\ldots,2^{n}$ we will proceed by double
induction with respect to $n$ and $1\leq i<2^{n}$. The inequality
holds for $n=1,2$ and let us assume that our theorem holds for $n$
and $1\leq i<2^{n}$. We consider two cases: $i$ is even and $i$ is
odd. If $i$ is even then $i=2k$ and we get
\begin{equation*}
e(2^{n+1}-i)=e(2^{n+1}-2k)=e(2^{n}-k)+1\leq n+1.
\end{equation*}
The last inequality follows from the induction hypothesis. For $i$
odd we have $i=2k-1$ and then
\begin{equation*}
e(2^{n+1}-i)=e(2^{n+1}-2k+1)=\op{max}\{e(2^{n}-k),e(2^{n}-k+1)\}\leq
n<n+1.
\end{equation*}
Thus we have that $e(2^{n+1}-i)\leq n+1$ and the result follows.

\bigskip

Finally, in order to prove (\ref{eidentity4}) we define
$u_{n}=\frac{4^{n+1}-1}{3}$. In order to show that
$\op{Mdeg}(n)=u_{n}$ it is enough to show that $e(u_{n})=n$ and
$e(u_{n}+i)>n$ for $i\in\N_{+}$.

First of all we note that $u_{0}=1$ and $u_{n+1}=4u_{n}+1$ for
$n\geq 1$. We prove that $e(u_{n})=n$ and $e(u_{n}+1)=n+1$. In order
to do this we use induction on $n$. These equalities are true for
$n=0,1$, and let us assume that are true for $n$. Then we have
\begin{equation*}
e(u_{n+1})=e(4u_{n}+1)=e(u_{n})+1=n+1
\end{equation*}
and
\begin{align*}
e(u_{n+1}+1)&=e(4u_{n}+2)=e(2u_{n}+1)+1=\op{max}\{e(u_{n}),e(u_{n}+1)\}+1\\
          &=\op{max}\{n,n+1\}+1=n+2
\end{align*}
and the result follows.

Now we prove that $e(u_{n}+i)>n$ for given $n$ and $i\in\N_{+}$.
This inequality is true for $n=0$ and $i\in\N_{+}$. If $i=4k$ then
we have
\begin{equation*}
e(u_{n+1}+i)=e(4u_{n}+4k+1)=e(u_{n}+k)+1>n+1.
\end{equation*}
If $i=4k+1$ then we have
\begin{align*}
e(u_{n+1}+i)&=e(4u_{n}+4k+2)=e(2u_{n}+2k+1)+1=\\
            &=\op{max}\{e(u_{n}+k),e(u_{n}+k+1)\}+1>\op{max}\{n,n\}+1=n+1.
\end{align*}
For $i=4k+2$ we get
\begin{equation*}
e(u_{n+1}+i)=e(4u_{n}+4k+3)=e(u_{n}+k+1)+1>n+1,
\end{equation*}
and finally for $i=4k+3$ we get
\begin{equation*}
e(u_{n+1}+i)=e(4u_{n}+4k+4)=e(2u_{n}+2k+2)+1=e(u_{n}+k+1)+2>n+2
\end{equation*}
and our theorem follows.
\end{proof}

\begin{cor}\label{ineqfore}
We have $\lfloor \frac{\lg n}{2} \rfloor \leq e(n)\leq \lg n$.
\end{cor}
\begin{proof}
This is a simple consequence of the identities obtained in Theorem
\ref{extremalofe}.
\end{proof}










Another interesting property of the sequence
$\{e(n)\}_{n=1}^{\infty}$ is contained in the following.

\begin{cor}
Let $k\in \N$ be given. Then
\begin{equation*}
\lim_{n\rightarrow+\infty}\frac{e(n)}{e(n+k)}=1.
\end{equation*}
\end{cor}
\begin{proof}
In order to prove the demanded equality let us note that from the
definition of the sequence $e(n)$ we easily deduce that for each
$n\in\N$ we have $|e(n+1)-e(n)|\leq 1$. Thus, for given positive $k$
we have by triangle inequality
\begin{equation*}
|e(n)-e(n+k)|\leq \sum_{i=1}^{k}|e(n+i-1)-e(n+i)|\leq k.
\end{equation*}
 Form this inequality we get that
 \begin{equation*}
\left| \frac{e(n)}{e(n+k)}-1\right| \leq \frac{k}{e(n+k)}.
 \end{equation*}
Because the fraction $k/e(n+k)$ tends to zero with
$n\rightarrow\infty$ the result follows.
\end{proof}

\begin{thm}\label{equalvaluese}
\begin{enumerate}

\item The set $\cal{E}:=\{n\in\N:\;e(n)=e(n+1)\}$ is infinite and has the following property: If $m\in\cal{E}$ then
$n=4m+1$ satisfy $e(n)=e(n+1)=e(n+2)$. On the other hand side: if
$e(n)=e(n+1)=e(n+2)$ for some $n$ then there exists an integer $m$
such that $n=4m+1$ and $m\in\cal{E}.$

\item There doesn't exist an integer $n$ such that
\begin{equation*}
e(n)=e(n+1)=e(n+2)=e(n+3)
\end{equation*}
\end{enumerate}
\end{thm}
\begin{proof}
The fact that the set $\cal{E}$ is infinite is easy. For example if
$n=2^{k}-2$ then from Corollary \ref{+-expression} we get that
$e(n)=e(2^{k-1}-1)+1=k-1$ and $e(n+1)=e(2^{k}-1)=k-1$. But it should
be noted that the set $\cal{E}$ contains infinite arithmetic
progressions. More precisely we have
\begin{equation*}
e(8m+1)=e(8m+2)\quad\mbox{and}\quad e(8m+6)=e(8m+7).
\end{equation*}
This property is a consequence of the identities
\begin{align*}
e(8m+1)&=e(4m)=e(m)+2,\\
e(8m+2)&=e(4m+1)+1=e(2m)+1=e(m)+2
\end{align*}
and
\begin{align*}
e(8m+6)&=e(4m+3)+1=e(2m+2)+1=e(m+1)+2,\\
e(8m+7)&=e(4(2m+1)+3)=e(2(2m+1)+2)=e(m+1)+2.
\end{align*}

Now if $m\in\cal{E}$ then $e(m)=e(m+1)$ and we have
\begin{align*}
&e(4m+1)=e(m)+1,\\
&e(4m+2)=e(2m+1)+1=\op{max}\{e(m),e(m+1)\}+1=e(m)+1,\\
&e(4m+3)=e(m+1)+1=e(m)+1,
\end{align*}
and we are done.

Suppose now that $e(n)=e(n+1)=e(n+2)$. We show that $n=4m+1$.

If $n=4m$ then $e(n)=e(4m)=e(m)+2$ and $e(n+1)=e(4m+1)=e(m)+1$ and
we get a contradiction.

If $n=4m+2$ then $e(n+1)=e(4m+3)=e(m+1)+1$ and
$e(n+2)=e(4m+4)=e(m+1)+2$ and we get a contradiction.

Finally, if $n=4m+3$ then $e(n)=e(4m+3)=e(m+1)+1$ and
$e(n+1)=e(4m+4)=e(m+1)+2$ and once again we get a contradiction.

We have shown that if $e(n)=e(n+1)=e(n+2)$ then $n=4m+1$ for some
$m\in\N$. If now $e(n)=e(n+1)=e(n+2)$ for $n=4m+1$ then
\begin{align*}
e(n)&=e(4m+1)=e(m)+1,\\
e(n+1)&=e(4m+2)=e(2m+1)+1=\op{max}\{e(m),e(m+1)\}+1,\\
e(n+2)&=e(4m+3)=e(m+1)+1.
\end{align*}
Thus we get that $e(m)=\op{max}\{e(m),e(m+1)\}=e(m+1)$ and the first
part of our theorem is proved.
\bigskip

In order to prove (2) we consider two cases: $n$ even and $n$ odd.
The case $n$ even is immediately ruled out by the result from (1).
If $n=2k+1$ and the equality $e(n)=e(n+1)=e(n+2)=e(n+3)$ holds then
\begin{equation*}
\op{max}\{e(k),e(k+1)\}=e(k+1)+1=\op{max}\{e(k+1),e(k+2)\}=e(k+2)+1.
\end{equation*}
From the second equality we deduce that $e(k+1)+1=e(k+2)$ and form
the third equality we get $e(k+1)=e(k+2)+1$ and we arrive at a
contradiction.
\end{proof}

In Corollary \ref{sumrows} we obtain a closed form of the sum
$\sum_{i=1}^{2^{n}}B_{i}(t)$. It is an interesting question if an
analogous sum for the sequence $\{e(n)\}_{n=1}^{\infty}$ can be
obtained. As we will see the answer to this question is positive. In
fact, we have the following result.

\begin{cor}
For $n\geq 1$ we have
\begin{equation}\label{sume}
\sum_{i=1}^{2^{n}}e(i)=\frac{1}{36}((6n-7)2^{n+2}+18n+27+(-1)^{n}).
\end{equation}
\end{cor}
\begin{proof}
Let us define $S(n)=\sum_{i=1}^{2^{n}}e(i)$ and let $T(n)$ denote
the right hand side of the identity (\ref{sume}). In order to get a
closed form of the sum $S(n)$ we proceed by induction on $n$. Note
that (\ref{sume}) holds for $n=1,2$. Let us assume that (\ref{sume})
holds for $n$ and $n+1$. We prove that the identity holds for $n+2$.
In order to do this we compute
\begin{align*}
S(n+2)&=\sum_{i=1}^{2^{n+1}}e(2i)+\sum_{i=1}^{2^{n}-1}e(4i+1)+\sum_{i=1}^{2^{n}-1}e(4i+3)\\
      &=\sum_{i=1}^{2^{n+1}}(e(i)+1)+\sum_{i=1}^{2^{n}-1}(e(i)+1)+\sum_{i=0}^{2^{n}-1}(e(i+1)+1)\\
      &=S(n+1)+2^{n+1}+S(n)+2^{n}-1+S(n)-e(2^{n})+2^{n}\\
      &=S(n+1)+2S(n)+2^{n+2}-n-1\\
      &=T(n+1)+2T(n)+2^{n+2}-n-1,
\end{align*}
where the last equality follows from the induction hypothesis. A
simple calculation shows that $T(n+2)=T(n+1)+2T(n)+2^{n+2}-n-1$ and
the result follows.
\end{proof}

\begin{cor}
For $n\geq 1$ we have
\begin{equation*}
\sum_{i=1}^{2^{n}}(-1)^{i}e(i)=\frac{1}{12}(2^{n+2}+6n-3+(-1)^{n+1}).
\end{equation*}
\end{cor}



\section{Special values of $B_{n}(t)$ and some of their
consequences}\label{sec5}

In this section we compute some special values of the polynomials
$B_{n}(t)$. Next, we use the computed values it in order to solve
some polynomial diophantine equations involving Stern polynomials.

We start with the following.

\begin{thm}\label{specialvalues}
We have the following:
\begin{enumerate}
\item If $i\in\{0,1\}$ then:
\begin{equation*}
B_{n}(t)\equiv i \pmod t  \Leftrightarrow n\equiv i\pmod 2.
\end{equation*}

\item If $i\in\{-1,0,1\}$ then:
\begin{equation*}
B_{n}(t)\equiv i \pmod {t+1}  \Leftrightarrow n\equiv i\pmod 3.
\end{equation*}

\item For each $n\in\N$ we have the following congruences:
\begin{equation*}
\begin{array}{l}
  B_{n}(t)\equiv s(n) \pmod {t-1}, \\
  B_{n}(t)\equiv n \pmod {t-2}.
\end{array}
\end{equation*}
\end{enumerate}

\end{thm}
\begin{proof}
Each case in our theorem can be proved with the help of mathematical
induction. However we use a different approach. Let us note that for
any integer $a$ we have the congruence $B_{n}(t)\equiv
B_{n}(a)\pmod{t-a}$. So, we see that in order to prove our theorem
it is enough to know the value of the polynomial $B_{n}(t)$ at
$t=0,-1,1,2$. We will compute these values with the help of the
generating function of the sequence $\{B_{n}(t)\}_{n=0}^{\infty}$.

We start with the evaluation of $B_{n}(t)$ at $t=0$
\begin{align*}
B(0,x)&=\sum_{n=1}^{\infty}B_{n}(0)x^n=x\prod_{i=0}^{\infty}(1+x^{2^{i+1}})\\
      &=\frac{x}{1+x}\prod_{i=0}^{\infty}(1+x^{2^i})=\frac{x}{1-x^2}=\sum_{i=0}^{\infty}x^{2i+1},
\end{align*}
where in the last equality we used the well known formula
$\prod_{i=0}^{\infty}(1+x^{2^i})=\frac{1}{1-x}$. Our computation
shows that $B_{n}(0)=0$ for $n$ even and $B_{n}(0)=1$ for $n$ odd.

Now we compute the value of $B_{n}(-1)$. Similarly as in the
previous case we use the generating function of the sequence
$\{B_{n}(t)\}_{n=0}^{\infty}$ . Before we do that let us note that
\begin{equation*}
1-x^{2^i}+x^{2^{i+1}}=\frac{1+x^{3\cdot2^{i}}}{1+x^{2^{i}}}\quad\mbox{for}\;i\in\N.
\end{equation*}
This identity implies that
\begin{align*}
B(-1,x)&=\sum_{n=1}^{\infty}B_{n}(-1)x^n=x\prod_{i=0}^{\infty}(1-x^{2^i}+x^{2^{i+1}})\\
      &=x\prod_{i=0}^{\infty}\left(\frac{1+x^{3\cdot2^i}}{1+x^{2^i}}\right)=x\frac{\frac{1}{1-x^3}}{\frac{1}{1-x}}=\sum_{i=0}^{\infty}x^{3i+1}-\sum_{i=1}^{\infty}x^{3i-1}.
\end{align*}
Comparing now these two expansions of $B(-1,x)$ we get that
$B_{n}(-1)=0$ for $n\equiv0\pmod{3}$, $B_{n}(-1)=1$ for
$n\equiv1\pmod{3}$ and $B_{n}(-1)=-1$ for $n\equiv-1\pmod{3}$.

The first congruence given in (3) is obvious due to the fact that
$B_{n}(1)=s(n)$. The second comes from the identity
\begin{equation*}
B(2,x)=x\prod_{n=0}^{\infty}(1+2x^{2^n}+x^{2^{n+1}})=x\prod_{n=0}^{\infty}(1+x^{2^n})^2=\frac{x}{(1-x)^2}=\sum_{n=1}^{\infty}nx^{n}.
\end{equation*}
Our theorem is proved.
\end{proof}

We use the above theorem to prove the following.

\begin{thm}\label{eqofzerodeg}
If $\op{deg}_{t}(B_{n+1}(t)-B_{n}(t))=0$ for some $n\in\N$ then
$n=2^{m}-2$ and $B_{n+1}(t)-B_{n}(t)=1$.
\end{thm}
\begin{proof}
First of all we observe that if $B_{n+1}(t)-B_{n}(t)$ is a constant
then $n$ is even. Indeed, let us suppose that $n=2m+1$ for some
$m\in\N$. Then $n\equiv 1,3$ or $5\pmod{6}$. Now we use the
characterization of values $B_{n}(0)$ and $B_{n}(-1)$ given in
Theorem \ref{specialvalues}.

If $n\equiv 1\pmod{6}$ then we have
\begin{equation*}
\begin{array}{lll}
  B_{n+1}(-1)-B_{n}(-1)&= & -1-1=-2 \\
  B_{n+1}(0)-B_{n}(0)&= & 0-1=-1,
\end{array}
\end{equation*}
and we get a contradiction with the condition
$\op{deg}_{t}(B_{n+1}(t)-B_{n}(t))=0$. If now $n\equiv 3\pmod{6}$
then we have
\begin{equation*}
\begin{array}{lll}
  B_{n+1}(-1)-B_{n}(-1)&= & 1-0=1 \\
  B_{n+1}(0)-B_{n}(0)&= & 0-1=-1,
\end{array}
\end{equation*}
and again we get a contradiction. Finally, if $n\equiv 5\pmod{6}$
then we have
\begin{equation*}
\begin{array}{lll}
  B_{n+1}(-1)-B_{n}(-1)&= & 0-(-1)=1 \\
  B_{n+1}(0)-B_{n}(0)&= & 0-1=-1,
\end{array}
\end{equation*}
and once again we get a contradiction. Our reasoning shows that if
the polynomial $B_{n+1}(t)-B_{n}(t)$ is constant then $n$ is even.

Now we show that if $\op{deg}_{t}(B_{n+1}(t)-B_{n}(t))=0$ then
$n\equiv2\pmod{4}$. Suppose that $n=4m$ for some $m$. Then we have
$B_{4m}(t)=t^2B_{m}$ and we get that $B_{4m+1}(t)-t^2B_{m}(t)=c$ for
some $c\in\Z$. If we now differentiate this relation with respect to
$t$ we get $B_{4m+1}'(t)-t(2B_{m}(t)+tB_{m}'(t))=0$ for all $\in\R$.
Taking now $t=0$ we get that $B_{4m+1}'(0)=0$ which is a
contradiction. This follows from the fact proved in \cite[Theorem
8]{Kla} that the sequence $B_{n}'(0)$ counts the number of 1's in
the standard Grey code for $n-1$ and thus it is nonzero for $n\geq
1$.

Now we are ready to finish the proof of our theorem. Let us suppose
that $n$ is the smallest integer not of the form $2^{k}-2$ with the
property $\op{deg}_{t}(B_{n+1}(t)-B_{n}(t))=0$. From the preceding
reasoning we know that $n=4m+2$ for some $m\in\N$ and there exists a
positive integer $k$ such that $2^{k}-2<4m+2<2^{k+1}-2$. This
implies that $2^{k-1}-2<2m+1< 2^{k}-2$. Now we have
\begin{align*}
B_{4m+3}(t)-B_{4m+2}(t)&=B_{2m+1}(t)+B_{2m}(t)-tB_{2m+1}(t)\\
                        &=B_{m}(t)+B_{m+1}(t)+tB_{m}(t)-tB_{m}(t)-tB_{m+1}(t)\\
                        &=B_{m}(t)+(1-t)B_{m+1}(t)=B_{2m+1}(t)-B_{2m+2}(t).
\end{align*}
This computation shows that the polynomial $B_{2m+2}(t)-B_{2m+1}(t)$
is constant. So we see that the number $n'=2m+1$ has the property
$\op{deg}_{t}(B_{n'+1}(t)-B_{n'}(t))=0$ and we have $n'<n$. Moreover
$n'$ is not of the form $2^k-2$ because $n'$ is odd. So we get a
contradiction with the assumption of minimality of $n=4m+2$.
\end{proof}

Now we easily deduce the following generalization of Theorem
\ref{eqofzerodeg}.

\begin{cor}
If $a\geq 2$ is an integer then the equation $B_{n+a}(t)-B_{n}(t)=c$
has no solutions in integers $n, c$.
\end{cor}
\begin{proof}
First of all let us note that if equation $B_{n+a}(t)-B_{n}(t)=c$
has a solution in $n, c$ then $c=a$. Indeed, this follows from the
fact that $B_{n+a}(2)-B_{n}(2)=a$. Now if $a$ is an even integer
then $n+a$ and $n$ have the same parity and $B_{n+a}(0)-B_{n}(0)=0$,
a contradiction. If $a$ is odd then for even $n$ we have that
$B_{n+a}(0)-B_{n}(0)=1<a$, which leads to contradiction. If now $n$
is odd then $B_{n+a}(0)-B_{n}(0)=-1<a$, and once again we arrive at
a contradiction. We thus proved that the equation
$B_{n+a}(t)-B_{n}(t)=c$ has not solutions in integers $n, c$, which
finishes the proof.
\end{proof}
\begin{thm}
If $\op{deg}_{t}(B_{n+1}(t)-B_{n}(t))=1$ for some $n\in\N$ then
$n=1$ and $B_{n+1}(t)-B_{n}(t)=t-1$.
\end{thm}
\begin{proof}
We consider two cases: $n$ even and $n$ odd.

If $n$ is even then $B_{n+1}(0)-B_{n}(0)=1-0=1$. So if the equality
$\op{deg}_{t}(B_{n+1}(t)-B_{n}(t))=1$ holds we have
$B_{n+1}(t)-B_{n}(t)=at+1$ for some $a\in\Z$. Putting now $t=2$ and
using part (2) of Theorem \ref{specialvalues} we get that
$B_{n+1}(2)-B_{n}(2)=n+1-n=1$. So we deduce that $2a+1=1$ and we get
$a=0$, a contradiction.

If $n$ is odd then $B_{n+1}(0)-B_{n}(0)=0-1=-1$. So if the equality
$\op{deg}_{t}(B_{n+1}(t)-B_{n}(t))=1$ holds we have
$B_{n+1}(t)-B_{n}(t)=at-1$ for some $a\in\Z$. Putting now $t=2$ and
using part (2) of Theorem \ref{specialvalues} we get that
$B_{n+1}(2)-B_{n}(2)=n+1-n=1$. So we deduce that $2a-1=1$ and we get
$a=1$. Because $n$ is odd we have $n=2m+1$ and thus
\begin{equation*}
B_{n+1}(t)-B_{n}(t)=B_{2m+2}(t)-B_{2m+1}(t)=(t-1)B_{m+1}(t)-B_{m}(t)=t-1.
\end{equation*}
Putting now $t=1$ we get $B_{m}(1)=0$. We know that $B_{m}(1)=s(m)$
is the Stern diatomic sequence and in particular $s(m)>0$ for
positive $m$. Thus we deduce that $m=0$, which implies $n=1$ and we
get the equality $B_{2}(t)-B_{1}(t)=t-1$. Our theorem is proved.
\end{proof}

Essentially the same method as in the proof of the Theorem
\ref{eqofzerodeg} can be used to characterize those integers for
which $\op{deg}_{t}(B_{n+1}(t)-B_{n}(t))=2$. However, because the
calculations are rather lengthy we leave the task of proving the
following theorem to the reader.

\begin{thm}
If $\op{deg}_{t}(B_{n+1}(t)-B_{n}(t))=2$ for some $n\in\N$ then
$n=2^m+1$ and then $B_{n+1}(t)-B_{n}(t)=t^2-t-1$ or $n=3\cdot2^m-2$
and then $B_{n+1}(t)-B_{n}(t)=-t^2+t+1$.
\end{thm}

\section{Problems and conjectures}\label{sec6}

In this section we state some problems and conjectures which are
related to the sequence of Stern polynomials or to the sequence of
their degrees. Based on extensive numerical computation with PARI we
state the following.

\begin{conj}
If $a\in\Q$ and there exists a positive integer $n$ such that
$B_{n}(a)=0$ then $a\in\{-1, -1/2, -1/3, 0\}$.
\end{conj}

Let us define
$\bar{B}_{n}(t)=t^{e(n)}B_{n}\left(\frac{1}{t}\right)$. A polynomial
$B_{n}(t)$ is reciprocal if $B_{n}(t)=\bar{B}_{n}(t)$. We define
\begin{equation*}
\cal{R}:=\{n:\;B_{n}(t)=\bar{B}_{n}(t)\}.
\end{equation*}

Let us recall that, as was proved in \cite[Theorem 2]{Kla}, if we
write
\begin{equation*}
B_{n}(t)=\sum_{l=0}^{e(n)}\left|\begin{array}{c}
                                  n-1 \\
                                  l
                                \end{array}\right|t^{l}
\end{equation*}
then the number $\left|\begin{array}{c}
                                  n-1 \\
                                  l
                                \end{array}\right|$
is the number of hyperbinary representations of $n-1$ containing
exactly $l$ digits 1. Thus, if $n\in\cal{R}$ then for each $l\leq
e(n)$ we have $\left|\begin{array}{c}
                                  n-1 \\
                                  l
                                \end{array}\right|$=$\left|\begin{array}{c}
                                  n-1 \\
                                  e(n)-l
                                \end{array}\right|$.
It is an interesting question if the set $\cal{R}$ can be
characterized in a reasonable way.

Let us note that if $m$ is odd and $m\in\cal{R}$ then then for all
$k\in\N$ we have $2^{k}m\in\cal{R}$. Thus we see that in order to
characterize the set $\cal{R}$ it is enough to characterize its odd
elements. All odd $n\in\cal{R}, n\leq 2^{17}$, are contained in the
table below.
\bigskip
\begin{equation*}
\begin{array}{l}
  \hline
 n\in\cal{R},\;n\leq 2^{17}  \\
  \hline
1, 3, 7, 9, 11, 15, 27, 31, 49, 59, 63, 123, 127, 135, 177, 201,
225, 251, 255, 287, 297, 363,\\
377, 433, 441, 507, 511, 567, 729, 855, 945, 961, 1019,  1023, 1401, 1969, 2043, 2047,\\
3087, 3135, 3143, 3449, 3969, 4017, 4091, 4095, 5929, 7545, 8113, 8187, 8191, 11327,\\
 15737, 16129, 16305, 16379, 16383, 27711, 28551, 28799,
29199, 32121, 32689, 32763, \\
32767, 36737, 57375, 60479, 64889,
65025, 65457, 65531,
65535, 99449, 121863,\\
126015, 127239, 130425, 130993, 131067, 131071\\
  \hline
\end{array}
\end{equation*}
\bigskip

It is an easy exercise to show that if $n=2^{m}-1$ for some $m\in\N$
or $n=2^{m}-5$ for $m\geq 3$ then the polynomial $B_{n}(t)$ is
reciprocal. Another infinite family of integers with this property
is $n=(2^{m}-1)^2$. It is natural to state the following.

\begin{prob}
Characterize the set
$\cal{R}:=\{n\in\N:\;B_{n}(t)=\bar{B}_{n}(t)\}$.
\end{prob}

During the course of the proof of the Theorem \ref{equalvaluese} we
noted that the set $\cal{E}=\{n:\;e(n)=e(n+1)\}$ contains infinite
arithmetic progressions. It is an interesting question whether other
infinite arithmetic progressions are contained in $\cal{E}$.

Let us define
\begin{equation*}
p_{n}:=u_{n-1}=\frac{4^{n}-1}{3},\quad
q_{n}:=\frac{5\cdot4^{n}-2}{3}=5p_{n}+1.
\end{equation*}
It is easy to see that $p_{1}=1$ and $p_{n+1}=4p_{n}+1$ for $n\geq
1$. Moreover we have $q_{1}=6$ and $q_{n+1}=4q_{n}+2$ for $n\geq 1$.
Now let $i\in\N_{+}$ and consider the arithmetic progressions
\begin{equation*}
U_{i}:=\{2^{2i+1}n+p_{i}:\;n\in\N_{+}\},\quad\quad
V_{i}:=\{2^{2i+1}n+q_{i}:\;n\in\N_{+}\}.
\end{equation*}
We will prove that $\bigcup_{i=1}^{\infty}(U_{i}\cup V_{i})\subset
\cal{E}$. Because $U_{i}\cap U_{j}=\emptyset$ for $i\neq j$ and the
same property holds for $V_{i}, V_{j}$, it is enough to show that
$U_{i}, V_{i}\subset \cal{E}$ for $i\in\N_{+}$. We will proceed by
induction on $i$. We start with $U_{i}$. We know that the set
$U_{1}$ is contained in $\cal{E}$. So let us suppose that
$U_{i}\subset \cal{E}$. We take an element of $U_{i+1}$ and get
\begin{equation*}
e(2^{2i+3}n+p_{i+1})=e(2^{2i+3}n+4p_{i}+1)=e(4(2^{2i+1}n+p_{i})+1)=e(2^{2i+1}n+p_{i})+1,
\end{equation*}
and
\begin{align*}
e(2^{2i+3}n&+p_{i+1}+1)=e(2^{2i+3}n+4p_{i}+2)=e(2^{2i+2}n+2p_{i}+1)+1\\
           &=\op{max}\{e(2^{2i+1}n+p_{i}), e(2^{2i+1}n+p_{i}+1)\}+1=e(2^{2i+1}n+p_{i})+1,
\end{align*}
where the last equality follows from the induction hypothesis. This
shows that $U_{i}\subset\cal{E}$.

Because exactly the same type of reasoning can be used to show that
$V_{i}\subset\cal{E}$, we leave the details to the reader.

We also check that the sequences
$\{2p_{n}\}_{n=1}^{\infty},\;\{q_{n}\}_{n=1}^{\infty}$ are contained
in $\cal{E}$. In order to show that the $2p_{n}\in\cal{E}$ for given
$n\in\N_{+}$ we proceed by induction. Clearly $2p_{1}=2\in\cal{E}$.
Let us suppose that for some $n$ the number $2p_{n}$ is an element
of $\cal{E}$, and thus $e(2p_{n})=e(2p_{n}+1)$. Then we have
\begin{align*}
e(2p_{n+1})&=e(2(4p_{n}+1))=e(4p_{n}+1)+1=e(p_{n})+2,\\
e(2p_{n+1}+1)&=e(4(2p_{n})+3)=e(2p_{n}+1)+1=e(2p_{n})+1=e(p_{n})+2.
\end{align*}

Using induction we prove that $q_{n}\in\cal{E}$ for any given $n$.
Indeed, for $n=1$ we have $q_{1}=6$ and $e(6)=e(7)$. Suppose that
$e(q_{n})=e(q_{n}+1)$ for some $n$. Then we have
\begin{align*}
e(q_{n+1})  &=e(4q_{n}+2)=e(2q_{n}+1)+1\\
            &=\op{max}\{e(q_{n}),e(q_{n}+1)\}+1=e(q_{n}+1)+1,\\
e(q_{n+1}+1)&=e(4q_{n}+3)=e(q_{n}+1)+1.
\end{align*}

Now we define the set
\begin{equation*}
\cal{E}':=\{2p_{n}\}_{n=1}^{\infty}\cup\{q_{n}\}_{n=1}^{\infty}\cup\bigcup_{i=1}^{\infty}(U_{i}\cup
V_{i})
\end{equation*}

Using a simple script written in PARI \cite{Pari} we calculated all
members of the set $\cal{E}$ for $n\leq 10^{8}$ and we checked that
all these numbers are contained in the set $\cal{E}'$. This leads us
to the following.

\begin{conj}
We have $\cal{E}=\cal{E}'$.
\end{conj}

\begin{conj}
Let $p$ be a prime number. Then the polynomial $B_{p}(t)$ is
irreducible.
\end{conj}

\noindent Using a simple script written in PARI we check that the
above conjecture is true for the first million primes.

\begin{conj}
For each $k\in\N_{+}$ there exists an integer $n$ with exactly $k$
prime divisors such that the polynomial $B_{n}(t)$ is irreducible.
\end{conj}

\noindent Let $k$ be positive integer and let $c_{k}$ be the
smallest integer such that $c_{k}$ has exactly $k$ prime divisors
and the polynomial $B_{c_{k}}(t)$ is irreducible. Below we tabulate
the values of $c_{k}$ for $k\geq 7$.

\begin{equation*}
\begin{array}{lll}
\hline
  k & c_{k} & \mbox{Factorization of}\; c_{k} \\
  \hline
  1 & 2 & 2 \\
  2 & 55 & 5\cdot11 \\
  3 & 665 & 5\cdot7\cdot19 \\
  4 & 6545 & 5\cdot7\cdot11\cdot17 \\
  5 & 85085 & 5\cdot7\cdot11\cdot13\cdot17 \\
  6 & 1616615 & 5\cdot7\cdot11\cdot13\cdot17\cdot19 \\
  7 & 37182145 & 5\cdot7\cdot11\cdot13\cdot17\cdot19\cdot23\\
  \hline
\end{array}
\end{equation*}





\bigskip \noindent {\bf Acknowledgments}: The Authors would like to thank the anonymous
referee for several helpful suggestions which significantly improved
the original presentation.

\bigskip

\noindent Jagiellonian University, Institute of Mathematics,
{\L}ojasiewicza 6, 30-348 Krak\'ow, Poland; e-mail:\; {\tt
maciej.ulas@uj.edu.pl}

 \end{document}